\documentclass{elsarticle}

\usepackage{caption}
\usepackage{amsmath}
\usepackage{amsthm}
\usepackage{amsfonts}
\usepackage{amssymb}
\usepackage{amscd}
\usepackage[usenames]{color}

\usepackage{calc}
\usepackage{subcaption}
\usepackage{bm}
\newtheorem{thm}{Theorem}

\newtheorem{lemma}[thm]{Lemma}
\newtheorem{prop}[thm]{Proposition}
\newtheorem{cor}[thm]{Corollary}

\newcommand{\Nb}{\mathbb{N}}

\newcommand{\Fb}{\mathbb{F}}
\newcommand{\Pb}{\mathbb{P}}

\newcommand{\PP}{\mathcal{P}}

\newcommand{\LL}{\mathcal{L}}

\newcommand{\GG}{\mathcal{G}}
\newcommand{\HH}{\mathcal{H}}
\newcommand{\Mc}{\mathcal{M}}

\newcommand{\GL}{\operatorname{GL}}

\newcommand{\Kerr}{\operatorname{Ker}}

\newcommand{\kMag}{\operatorname{-Mag}}
\newcommand{\rMag}{r\operatorname{-Mag}}
\newcommand{\X}{\boldsymbol{X}}
\newcommand{\bB}{\boldsymbol{b}}
\newcommand{\sB}{\boldsymbol{s}}

\begin{document}
\title{On Magic Finite Projective Space}
\markright{Magic Finite Projective Space}
\author{David Nash and Jonathan Needleman}
\begin{abstract}
This paper studies a generalization of magic squares to finite projective space $\Pb^n(q)$.  We classify at all functions from $\Pb^n(q)$ into a finite field where the sum along any $r$-flat is $0$.  In doing so we show connections to elementary number theory and the modular representation theory of $\GL(n,q)$.
\end{abstract}

\maketitle

\section{Introduction.}
This paper investigates a generalization of magic squares to finite projective space.  A magic square can be viewed as an injective $\Nb$-valued function on a square grid where summing along any horizontal, vertical or main diagonal line will produce the same sum.   Small \cite{Small} enlarged this definition of magic square by looking at $k$-valued functions on a square grid for some field $k$.  He also removed the restriction that the functions be injective.  In doing so, Small showed that the dimension of magic squares over a given field $k$ is independent of $k$ when $n\geq 5$ and depends on the characteristic of $k$ for small $n$.  The description of how the dimension is dependent on the characteristic is given in \cite{Hou}.  It is natural to extend Small's work to finite projective planes $\Pi$ of order $n$, as the plane naturally comes with lines to sum over.  In \cite{NN} we show that no non-constant functions $f:\Pi\rightarrow k$ exist unless the characteristic of $k$ divides $n$.


This paper extends our previous work to finite projective space.  For finite projective space we could require that all functions sum to the same value along each of the lines, however, there are other natural possibilities.  For a field  $k$ and a $k$-valued function $f$ on a finite projective space, we call $f$ $r$-magic if it is injective and the sum of $f$ over any $r$-dimensional subspace ($r$-flat) is a constant $c$ independent of the $r$-flat.  Such a $c$ we call the magic sum.  This definition is a little too restrictive because $r$-magic functions do not form an algebraic structure.  Instead we weaken our definition to $r$-pseudomagic functions by removing the injectivity condition.   These functions form a vector space and so will be easier to study.   This differs from the terminology of Smith because we want to reserve magic functions to align with the historic term where magic objects have distinct values.  We first will show for a projective space defined over a finite field $\Fb_q$, non-constant $r$-pseudomagic functions into a field $k$ exist only if $\operatorname{char}(k)|q$.   The remainder of the paper is dedicated to studying the vector space of $\Fb_q$-valued, $r$-pseudomagic functions, for projective space defined over $\Fb_q$.  We give a complete combinatorial classification all such functions as polynomials.

This problem can be viewed in terms of the $0$-$1$ incidence matrix, $A_r$, of a projective space, which encodes the incidence structure of the points and $r$-flats in the projective space.  Historically, when viewed over a characteristic $p$ field, this matrix plays a crucial role in coding theory, and its $p$-rank has been studied extensively.  For projective planes this rank is calculated by Goethals and Delsarte \cite{Goethals}.  The result was generalized independently (and using different techniques) to the $p$-rank of the incidence matrix for points and hyperplanes in projective space by Graham and MacWilliams \cite{Graham} and Smith \cite{Smith}.  Shortly thereafter Hamada generalized this further to the incidence matrix $A_r$ \cite{Hamada}.  This rank formula is now known as Hamada's formula.  Despite all this work being completed in the late 1960s it is still an open question to determine the $p$-rank of general incidence matrices for $r$-flats and $s$-flats of projective space.

In 2000, Bardoe and Sin \cite{Bardoe} used the modular representation theory of $\GL_n(q)$ in order to better understand, and reprove, Hamada's formula.  In doing so they give a combinatorial description of the representation of $\GL_n(q)$ acting on functions over projective space.  We use Bardoe and Sin's classification to study the $r$-pseudomagic functions, which happen to be closely related to $\Kerr A_{r}$.  This analysis helps give a geometric interpretation of Bardoe and Sin's combinatorial parameters of the representations.

\section{Finite projective space and psuedomagic functions.}\label{projspace}
In this section we set notation and introduce finite projective space.  Fix $q=p^t$ a prime power, $\Fb_q$ a finite field of $q$ elements, and $K$ an algebraic completion of $\Fb_q$.  We let $V(q)$ be an $n+1$ dimensional vector space over $\Fb_q$ with standard basis $e_0,\ldots, e_n$, and let $V=V(q)\otimes K$ be an extension of $V(q)$ to an $n+1$ dimensional vector space over $K$.   We define $\Pb^n(q)=(\LL_0,\ldots, \LL_{n-1})$ to be the projective space of dimension $n$ defined over the field $\Fb_q$ with $\LL_r$ the set of $r+1$-dimensional subspaces of $V(q)$ which we call \emph{$r$-flats}.  In the special case of $r=0$ we call $\PP=\LL_0$ the points of $\Pb^n(q)$.  Projective space $\Pb^n(q)$ comes equipped with an incidence relation.  Subspaces $W_1, W_2 \subset V(q)$ are said to be incident, if $W_1\subset W_2$ or $W_2\subset W_1$, in this case we write  $W_1\sim W_2$.

Our interest was sparked by a desire to understand ``labelings" of the points in $\Pb^n(q)$ which are magic for $r$-flats.  That is, we are looking to understand functions $f:\PP\rightarrow G$, for some Abelian group $G$, that satisfy two criteria:
\begin{enumerate}
    \item[1)] There is a $c\in G$ so that for all $W_r\in\LL_r$,  $\sum_{p\sim W_r} f(p)=c$.
    \item[2)] $f$ is injective.
\end{enumerate}

Functions satisfying both conditions we call \emph{$r$-magic}, and those that satisfy condition 1) we call \emph{$r$-pseudomagic}.  In \cite{NN} we show that there are no non-trivial (non-constant) 1-pseudomagic functions for finite projective planes $\Pb^2(q)$ if $G$ is a torsion-free group.  We further show that there are no non-trivial 1-pseudomagic functions for $\Pb^2(q)$ unless $G$  contains a cyclic group of order $p$.
 We will let $\rMag(\Pb^n(q),G)=\{f:\PP\rightarrow G \mid f \text{ is $r$-pseudomagic}\}$. The functions where the magic sum is $0$, are denoted by $\rMag_0(\Pb^n(q), G)$.

The rest of this section is dedicated to showing that $\rMag(\Pb^n(q), G)$ only contains constant functions, denoted by $\langle1\rangle$, if $G$ does not contain a subgroup of order $p$.   Note that we instead prove the equivalent statement that if $G$ does not contain a subgroup of order $p$, then $\rMag_0(\Pb^n(q), G)=\{0\}$.  To prove this result we use two main ideas.  First, we show the result is true for hyperplanes ($(n-1)$-flats), and then we extend the result to arbitrary $r$-flats.

\begin{lemma}\label{hyper_res}  Let $G$ be an Abelian group and let $n\geq 2$.  If $(n-1)\kMag(\Pb^n(q),G)=\langle1\rangle$  then $(n-1)\kMag(\Pb^m(q),G)=\langle1\rangle$ for all $m\geq n$. \end{lemma}
\begin{proof}
Let $f\in(n-1)\kMag(\Pb^m(q),G)$.  For any $n$-flat $W \subset \Pb^m(q)$, $W$ is isomorphic to $\Pb^n(q)$, 
hence $f$ restricted to $W$, $f|_W$, is the constant function $g$ for some $g \in G$. Let $W_1$, $W_2 \subset \Pb^m(q)$ be $n$-flats.  Let $p_i$ be a point in $W_i$ and let $V$ be any $n$-flat containing $p_1$, $p_2$.  Then $f|_{W_1}=f|_V=f|_{W_2}=g$ for some $g\in G$.  Since $W_1$, $W_2$ are arbitrary  $f=g$.
\end{proof}

\begin{prop}\label{hyper_mag}
If $G$ has no subgroups of order $p$ then $(n-1)\kMag(\Pb^n(q),G)=\langle1\rangle$ for all $n \geq 2$.
\end{prop}
\begin{proof}
We proceed by induction on $n$ knowing that the statement is true when $n=2$ by \cite{NN}.  Assume $(r-1)\kMag(\Pb^r(q),G)=\langle1\rangle$.  By Lemma~\ref{hyper_res} we have $(r-1)\kMag(\Pb^m(q),G)=\langle1\rangle$ for $m\geq r$.  Let $V\in \Pb^{r+1}(q)$ be an $(r-1)$-flat, and let
\[\LL_r(V)=\{W\in \Pb^{r+1}(q) \mid W \text{ is an  $r$-flat and } W\sim V\}\]
Note that $|\LL_r(V)|=q+1$ because through dualizing this is the same as the number of points on a line.  Observe, for $f\in\rMag(\Pb^{r+1}(q),G)$
\[\sum_{W\in\LL_r(V)} f(W)=\sum_{p\not\in V}f(p)+(q+1)\sum_{p\in V}f(p)=\sum_{p\in \Pb^{r+1}(q)}f(p)+q\sum_{p\in V}f(p)\]
Since $f$ is $r$-pseudomagic and the sum is in terms of $r$-flats it is independent of $V$.  Comparing to another $(r-1)$-flat $V'\in \Pb^{r+1}(q)$ we have
\[q\sum_{p\in V}f(p)=q\sum_{p\in V'}f(p).\]
Since $G$ contains no subgroup of order $p$ the map $h\mapsto qh$ for $h\in G$ has a trivial kernel, and so we conclude $\sum_{p\in V}f(p)$ is independent of the choice of $(r-1)$-flat $V$.  That is, $f\in (r-1)\kMag(\Pb^{r+1}(q),G)$ and so $f$ is a constant function.
\end{proof}

\begin{cor}
If $G$ has no subgroups of order $p$ then $\rMag_0(\Pb^n(q),G)=\{0\}$.
\end{cor}
\begin{proof}
This follows from Proposition~\ref{hyper_mag} applied to Lemma~\ref{hyper_res}.
\end{proof}

Note that for a field $K$ these results say that $\rMag(\Pb^n(q),K)=\langle 1 \rangle$ if the characteristic of $K$ does not divide $q$, so the interesting case  is when the characteristic of $K$ and $q$ are not relatively prime.  The remainder of this paper we study $\rMag(\Pb^n(q),\Fb_q)$ which for brevity we will just write as $\rMag(\Pb^n(q))$.  Similarly we let $\rMag_0(\Pb^n(q),\Fb_q)=\rMag_0(\Pb^n(q))$.

The incidence matrix of $\Pb^n(q)$ is closely related to $\rMag_0(\Pb^n(q))$.  Using the incidence relation we define a series of incidence matrices  $A_r$ for $0<r<n$  as follows.  The columns of $A_r$ are indexed by the points $\PP$ and the rows are indexed by the $r$-flats $\LL_r$.  Then for $p\in \PP$ and $W\in\LL_r$ the $(W, p)$ location is $1$ if $p$ and $W$ are incident, and $0$ otherwise.  By viewing any function $f:\PP \to \Fb_q$ as a column vector we may apply $A_r$ as an operator on functions
\[A_r: \{f:\LL_0 \rightarrow \Fb_q\} \longrightarrow \{g:\LL_r\rightarrow \Fb_q\},\]
\[\text{where~} (A_r \cdot f)(W)=\sum_{\substack{p\in \LL_0 \\ p\sim W}} f(p).\]
This perspective brings about the insight that $r$-pseudomagic functions (with magic sum zero) correspond to the kernels of the incidence matrices $A_r$, i.e. $\Kerr(A_r)=\rMag_0(\Pb^n(q))$.

\section{Polynomials, Representation Theory, and Kummer's Binomial Theorem.}
To study $r$-pseudomagic functions we first need to understand the functions  $K[\Pb^n(q)] = \{f:\PP\rightarrow K\}$.  These can be characterized as polynomials in $n+1$ variables with  certain restrictions.  We choose coordinates $X_0, \ldots, X_n$ for $V$ that are dual to the standard basis and look at which monomials $m\in K[X_0, \ldots, X_n]$ restrict to functions on $\Pb^n(q)$.  These monomials will then form a monomial basis for $K[\Pb^n(q)]$.  For a monomial to be well-defined on points in projective space it must be invariant under scalar multiplication on the coordinates.  By Fermat's little theorem this happens only when $q-1$ divides total degree of the monomial.

However, not all such monomials are unique when restricted to $\Pb^n(q)$.  The kernel of the restriction map is generated by $\langle X_i^q-X_i\rangle_{i=0}^n$, along with $\delta_0$, the characteristic function of $0\in V(q)$.   The polynomials $X_i^q-X_i$ correspond to the fact that $c^{q}=c$ for all $c\in \Fb_q$, while $\delta_0$ appears because projective space is defined for homogenous coordinates and is not defined at $0\in V(q)$.  To account for $\delta_0$ in the kernel we set $X_0^{q-1}\cdots X_n^{q-1}=0$.

Before moving forward we introduce multi-index notation to help simplify the notation.  Given a monomial $X_0^{b_0} \cdots X_n^{b_n}$, denote the set of exponents by $\bB = (b_0, \dots, b_n)$ and the monomial itself by $\X^{\bB}$.  The total degree of the monomial is then $|\bB| = \sum_{i=0}^n b_i$. Using multi-index notation we summarize our discussion above by describing a monomial basis $\Mc(q)$ for $K[\Pb^n(q)]$.
  \begin{equation}
  \Mc(q)=\{\X^{\bB} \mid 0\leq b_i\leq q-1, (q-1)| |\bB|, \bB \neq (q-1, \dots, q-1)\}
  \end{equation}
Our main goal  is to determine which polynomials in $K[\Pb^n(q)]$ are $r$-pseudomagic.  In order to do this we will make use of results from the representation theory of the group $\GG=\GL_{n+1}(q)$.

The group $\GG$ has a natural action on the vector space $V(q)$ which extends to an action on $\Pb^n(q)$.  We use the action on $\Pb^n(q)$ to define an action of $\GG$ on $K[\Pb^n(q)]$ by setting $(g\cdot f)(x)=f(g^{-1}x)$ for all $g\in \GG$, $f \in K[\Pb^n(q)]$, and $x\in \PP$.  
Since the action of $\GG$ on $\Pb^n(q)$ sends $r$-flats to $r$-flats, it follows that  $\rMag_0(\Pb^n(q))$ (and hence $\Kerr A_r$) is a sub-representation.  To gain information about the $r$-pseudomagic functions we study the structure of this representation.

To motivate the discussion we begin with a result from classical number theory.

\begin{thm}[Kummer's Theorem (1852)]\label{Kummer}
Let $r$ be the largest natural number so that $p^r$ divides ${n}\choose{k}$, then $r$ is the number of carries involved in the base $p$ sum of $k$ and $n-k$.
\end{thm}

The original proof is due to Kummer \cite{Kummer}.  We need a generalization of Kummer's Theorem to multinomial coefficients that is a result of Dickson \cite{Dickson}.

\begin{thm}\cite{Dickson}\label{Multi_Kummer}
Let $r$ be the largest natural number so that $p^r$ divides the multinomial coefficient ${k_1+\cdots + k_n}\choose{k_1, \ldots, k_n}$, then $r$ is the sum of all the carries involved in the base $p$ sum of $\sum_{i=1}^nk_i $.
\end{thm}

Understanding the representation generated by $\GG$ acting on $X_0^{q-1}$ involves a special case of Theorem~\ref{Multi_Kummer}.  For a generic $g\in \GG$ that sends $X_0$ to $\sum_{i=0}^n a_iX_i$ we find that
\begin{equation}\label{expand}
g\cdot X_0^{q-1}=\sum_{b_0+\cdots +b_n=q-1}   {{q-1}\choose{b_0, \ldots, b_n}} (a_0^{b_0}\cdots a_n^{b_n})X_0^{b_0}\cdots X_n^{b_n}
\end{equation}

As we allow $g$ to range over all of $\GG$, the only monomials $\X^{\bB}$ in (\ref{expand}) that appear in the generated space are precisely those where
\[p\not\left\vert {q-1}\choose{b_0, \ldots, b_n} \right.\]

By Theorem~\ref{Multi_Kummer} we can determine when $p$ will divide the multinomial coefficient by analyzing the carries in the base $p$ sum of $b_0 + \cdots + b_n$.  Hence, for any monomial $\X^{\bB} \in K[\Pb^n(q)]$ we let $b_i=b_{i,0}p^0+\cdots + b_{i,(t-1)}p^{t-1}$ be the base $p$ expansion of each $b_i$ and recursively define $r_{j+1}$ as the number so that
\begin{equation}\label{carry}\sum_{i=0}^nb_{i,j} + r_j=(p-1)+r_{j+1}p,\end{equation}
where $r_0=0$.   That is, $r_j$ is exactly the number of carries into the $p^j$ place of the base $p$ sum $b_0+\cdots +b_n$.  So for each monomial $\X^{\bB} \in K[\Pb^n(q)]$ we can associate the list of carries $(r_1, \ldots, r_{t})$.

In our special case above, where $\X^{\bB}$ is in the image of $X_0^{q-1}$ under the action of $\GG$, we have $|\bB| = q-1$ and hence $r_t=0$ as there is no carry into the $p^t$ place.   Now, by Theorem~\ref{Multi_Kummer}, the coefficient on $\X^{\bB}$ is not divisible by $p$ -- and hence $\X^{\bB}$ is in the representation generated by $X_0^{q-1}$ -- if and only if $(r_1, \ldots, r_t)=(0, \ldots, 0)$.

This example is also a special case of the representation theory of $\GL_{n+1}(q)$ acting on $K[\Pb^n(q)]$ developed by Bardoe and Sin in \cite{Bardoe}. To each monomial $\X^{\bB}\in K[\Pb^n(q)]$ we associate a sequence of $t$ natural numbers $\sB=(s_0, \ldots, s_{t-1})$ in the set
\[\HH=\{\sB \mid 1\leq s_i\leq n-1, ps_{i+1}-s_i\leq n(q-1), s_0=\cdots =s_{t-1}\neq n-1\}.\]

Let $\operatorname{Fr}_p$ be the Frobenius map which sends $a\in \Fb_q\mapsto a^p$.  Let $\X^{\bB}\in K[\Pb^n(q)]$.  Since $\operatorname{Fr}_p$ is a field automorphism it acts on $K[\Pb^n(q)]$.  Hence, $\operatorname{Fr}^{-i}_p(\X^{\bB})=\X^{\bB'}\in K[\Pb^n(q)]$, and so $|\bB'|=s_i(q-1)$ for some $s_i\in \Nb$ with $1\leq s_i<n$.  Then to $\X^{\bB}$ we associate the sequence $\sB=(s_i)_{i=0}^{t-1}$.

When $n\leq q-1$,  $ \sB$ can equivalently be defined as $s_i=r_i+1$ for $1 \leq i \leq t-1$ and $s_0=r_t+1$ with $r_i$ defined as in (\ref{carry}). So the set $\HH$ can be thought of as a parametrization of monomials in $K[\Pb^n(q)]$ which generalizes the data used in Kummer's theorem and its generalization.
 


The space $\HH$ also comes equipped with a partial order, $\sB \leq \sB'$ if $s_i\leq s_i'$ for all $0\leq i\leq t-1$.  This partial order allows one to define a ``nice'' collection of representations.  Let $$Y(\sB):=\langle\text{all monomials associated to } \sB'\leq \sB\rangle$$ and let $$\overline{Y}(\sB):=\langle\text{all monomials associated to } \sB'< \sB \rangle$$

The following theorem is a compilation of the results from Bardoe and Sin which we will need.
\begin{thm}\cite{Bardoe}\label{Bardoe}
For each non-constant monomial $\X^{\bB} \in K[\Pb^n(q)]$ there is an associated $\sB \in \HH$ and the following hold:
\begin{enumerate}
    \item $Y(\sB)$ and $\overline{Y}(\sB)$ are sub-representations of $K[\Pb^n(q)]$.  Furthermore, $L(\sB):=Y(\sB)/\overline{Y}(\sB)$ is irreducible.\label{irr_rep}
    \item The set $\{L(\sB) \mid \sB \in \HH\}$ contains all irreducible representations that appear in the composition series of $K[\Pb^n(q)]/\langle1\rangle$. \label{comp_ser}
    \item The space generated by the action of $\GG$ on $\X^{\bB}$ is $Y(\sB)$. \label{mono_gen}
\end{enumerate}
\end{thm}

Returning to the example of $\GG$ acting on $X_0^{q-1}$ we see that the $\sB$ associated to $X_0^{q-1}$ is $\sB=(1, \ldots, 1)$.  Since $\sB$ is the minimal element of $\HH$, we know by Theorem~\ref{Bardoe} that
    $$Y(\sB)=\langle\X^{\bB} \mid \text{The associated } \sB'=(1,\ldots, 1)\rangle.$$
This is the space generated by monomials where $r_0=\cdots = r_{t-1}=0$, which is what we found using Theorem~\ref{Multi_Kummer}.  In this light Theorem~\ref{Bardoe} can be thought of as a partial generalization of Kummer's Theorem, in that it describes when the coefficient of a monomial in the image $g \cdot \X^{\bB}$, 
is not divisible by $p$ in terms of base $p$ expansions of $\bB$.  Specifically, if $\X^{\bB} \in K[\Pb^n(q)]$ is acted on by $g\in \GG$, then $p$ does not divide the coefficient on a monomial $\X^{\bB'}$ in the image if and only if $\sB'\leq \sB$, where $\sB, \sB'\in \HH$ are associated to $\X^{\bB}$ and $\X^{\bB'}$ respectively.  This result will allow us to describe which polynomials are in $\rMag_0(\Pb^n(q))$ in terms of the parameters from $\HH$.

\section{Pseudomagic functions.}
We now show which functions in $K[\Pb^n(q)]$ are $r$-pseudomagic.

\begin{prop}For $0<r<n$ if $m \in K[\Pb^n(q)]$, is a monomial of degree $r(q-1)$, then $m \in \rMag_0(\Pb^n(q))$. \end{prop}
\begin{proof} Let $m=\X^{\bB}$ where $|\bB|= r(q-1)$.   Instead of viewing $m$ as a function on $\Pb^n(q)$ it will be easier to treat it as a function on $V(q)$, so  we write $\overline{m}$ for $m$ viewed as function on $V(q)$.   For $D\subset V(q)$ set $\overline{m}(D)=\sum_{d\in D}\overline{m}(d)$.  Let $W\in \LL_r$ be an $r$-flat of $\Pb^n(q)$ spanned by $\mathbf{x}_0, \dots, \mathbf{x}_r \in V(q)$ with coordinates $\mathbf{x}_i = [x_{i,0}, \dots, x_{i,n}]$ for each $0 \leq i \leq r$.  Since all $\X^{\bB} \in K[\Pb^n(q)]$ are constructed to be independent of the action of scalars, applying $\overline{m}$ to a 1-flat $\langle \mathbf{x}_i \rangle$ gives
\begin{eqnarray*}
 \overline{m}(\langle\mathbf{x}_i\rangle)&=&\sum_{\alpha\in \Fb_q^*} \overline{m}(\alpha \mathbf{x}_i)= \sum_{\alpha\in \Fb_q^*} \overline{m}(\mathbf{x}_i)\\&=&-\overline{m}(\mathbf{x}_i)=-m(\mathbf{x}_i) \qquad \text{for each $i$.}
\end{eqnarray*}
Since we want $m(W)=0$ for all $W\in\LL_r$, we just need to show $\overline{m}(W)=0$.
\begin{align}
    \overline{m}(W) &= \sum_{(\alpha_0, \dots, \alpha_r)\in \Fb^{r+1}_q} \overline{m}(\alpha_0 \mathbf{x}_0+ \dots + \alpha_r\mathbf{x}_r)\notag\\
    &= \sum_{(\alpha_0,\dots, \alpha_r) \in \Fb^{r+1}_q} \left(\prod_{j=0}^n \left(\sum_{i=0}^r \alpha_i x_{i,j}\right)^{b_j}\right)\label{longform}
\end{align}
Viewing (\ref{longform})  as a polynomial in $x_{i,j}$, the coefficient of $\prod_{j=0}^n \prod_{i=0}^r x_{i,j}^{a_{i,j}}$ is
\begin{equation} \label{coef}
    \sum_{(\alpha_0, \dots, \alpha_r)\in \Fb^{r+1}_q}\left(\prod_{j=0}^n {{b_j}\choose{a_{0,j}, \dots, a_{r,j}}}\right) \prod_{i=0}^r \alpha_i^{(a_{i,0}+\cdots +a_{i,n})},
\end{equation}
where $0\leq a_{i,j} \leq b_j$ and $\sum_{i=0}^r a_{i,j} = b_j$ for all $j$.

We now show that each of these coefficients is zero.  It is a well known fact that $\sum_{\alpha \in \Fb_q} \alpha^k = 0$ if and only if $(q-1) {\not |}~ k$.  Assume there is an $i$ so that $(q-1)$ does not divide $A=\sum_{j=0}^n a_{i,j}$.  Without loss of generality let $i=0$, then the coefficient becomes
   \begin{equation} \label{nodiv}
    \sum_{(\alpha_1, \dots, \alpha_r)\in \Fb^{r}_q}C\left(\prod_{i=1}^r \alpha_i^{(a_{i,0}+\cdots +a_{i,n})}\right)\left(\sum_{\alpha_0\in \Fb_q} \alpha_0^A\right)=0,
   \end{equation}
where \[C= \left(\prod_{j=0}^n {{b_j}\choose{a_{0,j}, \dots, a_{r,j}}}\right).\]
Hence, we need only to consider the cases where $(q-1)$ divides $\sum_{j=0}^n a_{i,j}$ for all $i$.  Since the total degree of our polynomial is equal to $r(q-1)$ it follows there exists some $i$,  $0 \leq i \leq r$ such that $\sum_{j=0}^n a_{i,j} = 0$.  Suppose, without loss of generality, that $\sum_{j=0}^n a_{0,j} = 0$.  Then our coefficient becomes
$$\sum_{\alpha_0 \in \Fb_q} \left(\sum_{(\alpha_1, \dots, \alpha_r)\in \Fb^r_q}\left(\prod_{j=0}^n {{b_j}\choose{a_{0,j}, \dots, a_{r,j}}}\right) \prod_{i=1}^r \alpha_i^{(a_{i,1}+a_{i,2}+\cdots +a_{i,n})}\right).$$
The interior portion is constant with respect to the choice of $\alpha_0$, hence we are summing a constant $q$ times, which gives 0 in $\Fb_q$.  Since $m$ is identically zero when summed over any $r$-flat $W$, it follows that $m \in \rMag_0(\Pb^n(q))$ as desired.
\end{proof}

Polynomials of degree $r(q-1)$ are not all the $r$-pseudomagic functions, but in some sense all of them are generated from these functions.  The Frobenius map $\text{Fr}_p:x\mapsto x^p$ is a field automorphism, and thus composing it with any $r$-pseudomagic function will result in another $r$-pseudomagic function.  By construction the Frobenius map cyclically permutes the sequence $\sB$.  For $m\in K[\Pb^n(q)]$ a monomial with degree $r(q-1)$ and associated sequence $\sB \in\HH$, we have by construction of $\sB$ that $s_0=r$.  Applying the Frobenius map to $m$ will thus result in a monomial $m'$ that has $r$ in the associated sequence $\sB' \in \HH$.  Furthermore, if $m \in \rMag_0(\Pb^n(q))$, then $m \in r'\operatorname{-Mag}_0(\Pb^n(q))$ for all $r< r' < n$ as well.  This follows because each point in a fixed $r'$-flat $W$ will be incident to exactly $\ell$ $r$-flats in $W$, where $\ell \equiv 1\mod p$.  So  summing over all $r$-flats within a given $r'$-flat is the same as summing over the $r'$-flat.  We can sum up our discussion with the following corollary:


\begin{cor}\label{r_in_s}
If $m \in K[\Pb^n(q)]$ is a monomial associated to $\sB \in \HH$ and there is an $i$ so that $s_i \leq r$, then $m\in \rMag_0(\Pb^n(q))$.
\end{cor}

To finish we show the condition in Corollary~\ref{r_in_s} is sufficient.  That is, if $m \in K[\Pb^n(q)]$ is a monomial associated to $\sB \in \HH$ and $s_i>r$ for all $0\leq i\leq t-1$ then  $m \not\in \rMag_0(\Pb^n(q))$.

We analyze the monomial $m_r=X_0^{q-1}\cdots X_r^{q-1}$ which is associated to $(r+1,\ldots, r+1)\in \HH$.   Consider the $r$-flat $W\subset V(q)$ generated by $e_0, \ldots, e_r$.  For $v=\alpha_0e_0+\cdots + \alpha_re_r\in W$ with $\alpha_i\in \Fb_q$ we have
\[m_r(v)=
\begin{cases} 0 & \text{if } \alpha_j=0 \text{ for some j}\\
1 & \text{else}
\end{cases}\]
and so we have
 \[m_r(W)=-\sum_{(\alpha_0,\ldots, \alpha_r)\in (\Fb_q^*)^{r+1}} 1=-|(\Fb_q^*)^{r+1}|=-(q-1)^{r+1}\neq 0\in \Fb_q\]
Hence $m_r\not\in \rMag_0(\Pb^n(q))$.  However, it is not in $\rMag(\Pb^n(q))$ either, since $m_r$ is zero on every point in the $r$-flat $W'$ generated by $e_1, \ldots, e_{r+1}$ since the coefficient of $e_0$ is always $0$ for $v\in W'$.

Now let $\sB\in \HH$ with $s_i>r$ for all $i$, and let $m$ be a monomial associated to $\sB$.  If we assume that $m\in \rMag_0(\Pb^n(q))$ then 
by Theorem~\ref{Bardoe}, $Y(\sB)\subset \rMag_0(\Pb^n(q))$.  This implies that $Y(r+1, \ldots, r+1)\subset Y(\sB)\subset \rMag_0(\Pb^n(q))$ and hence $m_r\in \rMag_0(\Pb^n(q))$, which we just showed to be false.  Summarizing our discussion we have our main theorem characterizing $\rMag_0(\Pb^n(q))$.

\begin{thm}\label{Main}
If $m\in K[\Pb^n(q)]$ is a monomial associated to $\sB\in \HH$, then $m \in \rMag_0(\Pb^n(q))$ if and only if $\operatorname{min}\{s_0, \dots, s_n\}\leq r$.
\end{thm}

\section{Conclusions and further questions.}
Theorem~\ref{Main} helps give a geometric interpretation to the combinatorial set $\HH$.  The minimum of the terms in $\sB\in \HH$ gives geometric information about the functions $f\in Y(\sB)$.  Since $K[\Pb^n(q)]$ is a geometric object we expect there is a way to geometrically distinguish representations $Y(\sB)$ and $Y(\sB')$ when $\sB$ and $\sB'$ have the same minimum.

We would also like to see a deeper connection  of the set $\HH$ to a generalization Kummer's theorem.  As of now we can only use $\HH$ to describe when a prime $p$ divides a coefficient  of a term in the image of a monomial under the action of $\GG$.  Given that the data $\HH$ provides is closely related to what is needed for  Kummer's theorem we expect that one can expand on Theorem~\ref{Bardoe} to determine the exact power of $p$ dividing a coefficient  of a term in the image of a monomial under the action of $\GG$.

Finally, our original goal was to find functions which are $r$-magic for projective space (injective $r$-pseudomagic functions).  We know there are no $r$-magic functions into $\Fb_q$ since $|\Pb^n(q)|=(q^{n+1}-1)/(q-1)$, but there is a chance that such a function exists mapping into $\Fb_q^n$.  That is, it is theoretically possible for $[\rMag_0(\Pb^n(q))]^n$ to contain an injective function.  It is our hope that our main result, Theorem~\ref{Main}, which characterizes $\rMag_0(\Pb^n(q))$ in terms of polynomials will help lead to a concrete construction of such a function.

\renewcommand{\bibname}{\textsc{references}} 
\bibliographystyle{elsart-num-sort}	
\bibliography{references}	

\def\cprime{$'$} \def\cprime{$'$}
\begin{thebibliography}{10}
\expandafter\ifx\csname url\endcsname\relax
  \def\url#1{\texttt{#1}}\fi
\expandafter\ifx\csname urlprefix\endcsname\relax\def\urlprefix{URL }\fi

\bibitem{Bardoe}
M.~Bardoe, P.~Sin, The permutation modules for {${\rm GL}(n+1,{\bf F}_q)$}
  acting on {${\bf P}^n({\bf F}_q)$} and {${\bf F}^{n-1}_q$}, J. London Math.
  Soc. (2) 61~(1) (2000) 58--80.
\newline\urlprefix\url{http://dx.doi.org/10.1112/S002461079900839X}

\bibitem{Dickson}
L.~E. Dickson, Theorems on the residues of multinomial coefficients with
  respect to a prime modulus, Quart. J. Pure Appl. Math. 33 (1902) 378--384,
  jFM:33.0203.02.

\bibitem{Goethals}
J.-M. Goethals, P.~Delsarte, On a class of majority-logic decodable cyclic
  codes, IEEE Trans. Information Theory IT-14 (1968) 182--188.

\bibitem{Graham}
R.~L. Graham, J.~MacWilliams, On the number of information symbols in
  difference-set cyclic codes, Bell System Tech. J. 45 (1966) 1057--1070.

\bibitem{Hamada}
N.~Hamada, The rank of the incidence matrix of points and {$d$}-flats in finite
  geometries, J. Sci. Hiroshima Univ. Ser. A-I Math. 32 (1968) 381--396.

\bibitem{Hou}
X.-d. Hou, A.~G. Lecuona, G.~L. Mullen, J.~A. Sellers, On the dimension of the
  space of magic squares over a field, Linear Algebra Appl. 438~(8) (2013)
  3463--3475.
\newline\urlprefix\url{http://dx.doi.org/10.1016/j.laa.2012.12.014}

\bibitem{Kummer}
E.~E. Kummer, Uber die erg\:anzungss\:atze zu den allgemeinen
  reciprocit\:atsgesetzen, J. f\"ur die reine und angewandte Math. 44 (1852)
  93--146.
\newline\urlprefix\url{http://eudml.org/doc/147500}

\bibitem{NN}
D.~Nash, J.~Needleman, When are finite projective planes magic?, preprint.

\bibitem{Small}
C.~Small, Magic squares over fields, Amer. Math. Monthly 95~(7) (1988)
  621--625.
\newline\urlprefix\url{http://dx.doi.org/10.2307/2323304}

\bibitem{Smith}
K.~J.~C. Smith, On the {$p$}-rank of the incidence matrix of points and
  hyperplanes in a finite projective geometry, J. Combinatorial Theory 7 (1969)
  122--129.

\end{thebibliography}

\end{document}